\documentclass{article}

\pdfoutput=1

\usepackage{amsmath, amssymb, amsthm, enumerate, fullpage, bm}
\usepackage{verbatim, enumitem, bbm, etoolbox}
\usepackage{latexsym}

\apptocmd{\lim}{\limits}{}{}

\theoremstyle{definition}
\newtheorem{thm}{Theorem}
\newtheorem{theorem}[thm]{Theorem}

\newtheorem{lemma}[thm]{Lemma}

\numberwithin{subcase}{case}

\theoremstyle{definition}
\newtheorem{definition}[thm]{Definition}

\def\forkindep{\mathrel{\raise0.2ex\hbox{\ooalign{\hidewidth$\vert$\hidewidth\cr\raise-0.9ex\hbox{$\smile$}}}}}

\def\Ind{\setbox0=\hbox{$x$}\kern\wd0\hbox to 0pt{\hss$\mid$\hss}
	\lower.9\ht0\hbox to 0pt{\hss$\smile$\hss}\kern\wd0}

\def\Notind{\setbox0=\hbox{$x$}\kern\wd0\hbox to 0pt{\mathchardef
		\nn=12854\hss$\nn$\kern1.4\wd0\hss}\hbox to 0pt{\hss$\mid$\hss}\lower.9\ht0
	\hbox to 0pt{\hss$\smile$\hss}\kern\wd0}

\def\phi{\varphi}

\def\<{\langle}
\def\>{\rangle}

\begin{document}	

	\bibliographystyle{plain}
	
	\author{Douglas Ulrich
	\thanks{Partially supported
by Laskowski's NSF grant DMS-1308546.}\\
Department of Mathematics\\University of Maryland}
	\title{Keisler's Order is Not Linear, Assuming a Supercompact}
	\date{\today} 
	
	\maketitle
	
\begin{abstract}
We show that if there is a supercompact cardinal then Keisler's order is not linear.
\end{abstract}

Keisler's order is a partial order $\trianglelefteq$ defined on complete countable theories, introduced by Keisler in \cite{Keisler}.  It is defined by setting $T_1 \trianglelefteq T_2$ if and only if for all $\lambda$ and for all regular ultrafilters $\mathcal{U}$ on $\lambda$ and for all $M_i \models T_i$, if $M_2^\lambda/\mathcal{U}$ is $\lambda^+$-saturated then so is $M_1^\lambda/\mathcal{U}$. The regularity assumption shows that this does not depend on the choice of the models $M_i$, so we can say that $\mathcal{U}$ $\lambda^+$-saturates $T$ if $M^\lambda/\mathcal{U}$ is $\lambda^+$-saturated for any $M \models T$; then we can rephrase Keisler's order as saying $T_1 \trianglelefteq T_2$ if and only if for all $\lambda$ and for all regular ultrafilters $\mathcal{U}$ on $\lambda$, if $\mathcal{U}$ $\lambda^+$-saturates $T_2$ then $\mathcal{U}$ $\lambda^+$-saturates $T_1$.

In his original paper \cite{Keisler}, Keisler proved that there is a maximal class in Keisler's order; then in \cite{ShelahIso} Shelah obtained several results on Keisler's order, including determining the least two classes of Keisler's order (namely the theories without the finite cover property, and the stable theories with the finite cover property). After this progress was slow, until a recent spate of work by Malliaris and Shelah. Among other things in \cite{InfManyClass} they prove that Keisler's order has infinitely many classes, and in \cite{Optimals} they prove that if there is a supercompact cardinal then simplicity is a dividing line in Keisler's order, that is if $T$ is simple and $T' \trianglelefteq T$ then $T'$ is simple. 

Using the technology developed in these papers, we observe that under the existence of a supercompact cardinal, Keisler's order is not linear. \footnote{In private communications, I have learned that Malliaris and Shelah had independently obtained this result, but it was not disseminated.} More specifically, let $T_{n, k}$ be the theory of the generic $n$-clique free $k$-ary graph for any $n > k \geq 3$ (these are the theories used in \cite{InfManyClass} to get infinitely many classes), and let $T_{cas}$ be the simple non-low theory described by Casanovas in \cite{Casanovas}. Then we show that $T_{cas} \not \trianglelefteq T_{n, k}$ always, and if there is a supercompact cardinal then $T_{n,k} \not \trianglelefteq T_{cas}$. 

We first recall the general setup introduced in \cite{DividingLine}. Namely, let $\mathcal{B}$ be a complete Boolean algebra, let $T$ be a complete countable theory, let $\lambda$ be a cardinal and let $\mathcal{U}$ be an ultrafilter on $\mathcal{B}$. Then $\mathcal{U}$ is $(\lambda, \mathcal{B}, T)$-moral if whenever $(\mathbf{a}_s: s \in [\lambda]^{<\aleph_0})$ is a certain kind of distribution from $\mathcal{U}$ (namely, a $(\mathcal{B}, T, \overline{\phi})$-possibility for some sequence of formulas $\overline{\phi}$) then $(\mathbf{a}_s)$ has a multiplicative refinement in $\mathcal{U}$. This definition is actually a generalization of $\lambda^+$-saturation, that is: suppose $\mathcal{B} = \mathcal{P}(\lambda)$ and $\mathcal{U}$ is $\lambda$-regular. Then $\mathcal{U}$ $\lambda^+$-saturates $T$ if and only if $\mathcal{U}$ is $(\lambda, \mathcal{B}, T)$-moral.

Also, given $\mu \geq \theta$ with $\theta$ regular and $\mu = \mu^{<\theta}$, and given a set $X$, let $P_{X, \mu, \theta}$ be the set of all partial functions from $X$ to $\mu$ of cardinality less than $\theta$, ordered by reverse inclusion; let $\mathcal{B}_{X, \mu, \theta}$ be the Boolean algebra completion of $P_{X, \mu, \theta}$. So $\mathcal{B}_{X, \mu, \theta}$ has the $\mu^+$-c.c. and is $<\theta$-distributive. For each $f \in P_{X, \mu, \theta}$ let $\mathbf{x}_f$ be the corresponding element of $\mathcal{B}_{X, \mu, \theta}$.

The key fact about this setup is Theorem 6.13 from \cite{DividingLine}:

\begin{theorem}
Suppose $T_0$, $T_1$ are complete countable theories. Suppose there are $\lambda \geq \mu \geq \theta$ and an ultrafilter $\mathcal{U}$ on $\mathcal{B} := \mathcal{B}_{2^\lambda, \mu, \theta}$ such that $\mathcal{U}$ is $(\lambda, \mathcal{B}, T_1)$-moral but not $(\lambda, \mathcal{B}, T_0)$-moral. Then $T_0 \not \trianglelefteq T_1$.
\end{theorem}

Actually, with a little cardinal arithmetic one can show that this holds if $\mathcal{B}_{2^\lambda, \mu, \theta}$ is replaced by any complete Boolean algebra $\mathcal{B}$ with $|\mathcal{B}| \leq 2^\lambda$, but we won't need this.

We now indicate the various facts from \cite{DividingLine}, \cite{Optimals}, and \cite{InfManyClass} we will need to get our theorem. First, we say that the complete countable theory $T$ is low if it is simple and for every formula $\phi(x, \overline{y})$, there is some $k$ such that for all $\overline{b}$, if $\phi(x, \overline{b})$ does not $k$-divide over $\emptyset$ then it does not divide over $\emptyset$. This is the standard definition of low, for instance it is equivalent to the definition in \cite{Buechler}. Malliaris defined low slightly differently in \cite{MalliarisFlex}, namely not requiring $T$ to be simple. To clarify, let us say that $T$ has the finite dividing property if there is some formula $\phi(x, \overline{y})$ such that for every $k$ there is some indiscernible sequence $(\overline{b}_n: n < \omega)$ over the emptyset such that $\{\phi(x, \overline{b}_n: n < \omega)\}$ is $k$-consistent but not consistent. What Malliaris calls low is what we call not having the finite dividing property; and we say that a theory $T$ is low if it is simple and does not have the finite dividing property. 

One reason to prefer our definition is the following: it is easy to check that $(\mathbb{Q}, <)$ does not have the finite dividing property, and so the finite dividing property is not a dividing line in Keisler's order. On the other hand, in future work we will show that lowness is a dividing line in Keisler's order.

The following theorem is a special case of Conclusion 9.10 from \cite{DividingLine}:

\begin{theorem}\label{Th1}
Let $T$ be any non-low theory and let $\lambda > \aleph_0$. Then no ultrafilter $\mathcal{U}$ on $\mathcal{B} := \mathcal{B}_{2^\lambda, \aleph_0, \aleph_0}$ is $(\lambda, \mathcal{B}, T)$-moral.
\end{theorem}

Now fix $n > k \geq 3$ for the rest of this paper; so $T_{n,k}$ is the theory of the random $n$-clique free $k$-ary graph. (In \cite{InfManyClass}, this theory is referred to as $T_{n-1, k-1}$, but we stick to the more common notation. In particular this affects the indicing for all the theorems we quote.) We will need the following two theorems from \cite{InfManyClass}. The following is (a special case of) Theorem 4.1.

\begin{theorem}\label{Th2}
Write $\lambda = \aleph_{k-2}$. There is an ultrafilter $\mathcal{U}$ on $\mathcal{B} := \mathcal{B}_{2^\lambda, \aleph_0, \aleph_0}$ which is $(\lambda, \mathcal{B}, T_{n, k})$-moral.
\end{theorem}

The following is not literally a special case of Claim 5.1 from \cite{InfManyClass}, but it has exactly the same proof. It would be a special case if we had $\sigma = \aleph_0$, but all that is used in the proof is that the revelant algebra $\mathcal{B}$ has the $\sigma^+$-c.c., which follows from $\sigma = \sigma^{<\sigma}$. 

\begin{theorem}\label{Th3}
Suppose $\sigma = \sigma^{<\sigma}$. Write $\lambda = \sigma^{+(n-1)}$. Then no ultrafilter $\mathcal{U}$ on $\mathcal{B} := \mathcal{B}_{2^\lambda, \sigma, \sigma}$ is $(\lambda, \mathcal{B}, T_{n, k})$-moral.
\end{theorem}

Thus, by Theorem~\ref{Th1} and Theorem~\ref{Th2}, we have (in just ZFC) that if $T$ is any non-low theory then $T \not \trianglelefteq T_{n, k}$. (This is a special case of the fact that low theories form a dividing line in Keisler's order, as mentioned previously.) Going forward, we take the most natural example of a simple non-low theory $T_{cas}$, and show that if there is a supercompact cardinal $\sigma$, then $T_{n,k} \not \trianglelefteq T_{cas}$. By Theorem~\ref{Th3} note it will suffice to show that if $\sigma$ is supercompact and $\lambda \geq \sigma$, then there is an ultrafilter on $\mathcal{B}_{2^\lambda, \sigma \sigma}$ which is $(\lambda, \mathcal{B}, T_{cas})$-moral.

$T_{cas}$ was introduced by Casanovas \cite{Casanovas} and was in fact the first example of a simple nonlow theory. The language $\mathcal{L}_{cas}$ is $(R, P, I, I_n: 1 \leq n < \omega)$, where $P, I, I_n$ are each unary relation symbols and $R$ is binary. We adopt the convention that $a, a',\ldots$ are elements of $P$, $b, b', \ldots$ are elements of $I$.

\begin{enumerate}
\item The universe is the disjoint union of $P$ and $I$, both infinite;
\item Each $I_n \subseteq I$, and the $I_n$'s are infinite and disjoint;
\item $R \subseteq P \times I$;
\item For each $a \in P$ and for each $n < \omega$, there are exactly $n$ elements $b \in I_n$ such that $R(a,b)$;
\item Whenever $B_0, B_1$ are finite disjoint subsets of $I$ such that each $|B_1 \cap I_n| \leq n$, there is $a \in P$ such that $R(a,b)$ for all $b \in B_1$ and $\lnot R(a, b)$ for all $b \in B_0$.
\item For all $A_0, A_1$ finite disjoint subsets of $P$, there is $b \in I$ such that $R(a, b)$ for all $a \in A_1$ and $\lnot R(a, b)$ for all $a  \in A_0$.
\end{enumerate}

Actually, if in the definition of $\mathcal{L}_{cas}$ we allow $n = 0$ then $I_0$ will be completely harmless, so for notational convenience we let $\mathcal{L}_{cas}$ be $(R, P, I_n: n < \omega)$.

In \cite{Casanovas} it is shown that $T_{cas}$ is complete, and is the model companion of the theory axiomatized by the first four items above. In particular it is shown that $T_{cas}$ has quantifier elimination in an expanded language, where we add predicates $S_{\ldots}$ that express the following: given $A_0, A_1 \subset P$ finite disjoint with $A_0 \not= \emptyset$, how many $b \in I_n$ are there such that $R(a, b)$ for all $a \in A_1$ and $\lnot R(a, b)$ for all $a \in A_0$. Thus the algebraic closure of a set $X$ is $X \cup \bigcup\{b \in \bigcup_n I_n: \mbox{ there is } a \in X \cap P \mbox{ such that } R(a, b)\}$, and every formula over a set $X$ is equivalent to a quantifier-free formula over $\mbox{acl}(X)$.

Casanovas also shows that $T_{cas}$ is simple with the following forking relation: $X \forkindep_{Z} Y$ iff $\mbox{acl}(X) \cap \mbox{acl}(Y) \subseteq \mbox{acl}(Z)$. Clearly also the formula $R(x, y)$ witnesses that $T_{cas}$ is not low. In future work we will show that $T_{cas}$ is a minimal nonlow theory in Keisler's order.

We will want the following lemma, which follows immediately from the quantifier elimination in the expanded language discussed above:

\begin{lemma}\label{TCasLemma1}
Let $M \models T_{cas}$. As notation let $T_\omega$ denote $I \backslash \bigcup_n I_n$.
\begin{itemize} 
\item For each $n < \omega$, there is a unique nonalgebraic type $p(x)$ over $M$ with $I_{n}(x) \in p(x)$. It is isolated by the formulas $I_n(x)$ together with $\lnot R(a, x)$ for each $a \in P^M$.
\item For each $A \subseteq P^M$ let $p_A(x)$ be the type over $M$ that says $I_\omega(x)$ holds, $x \not= b$ for each $b \in I^M$, and finally for each $a \in P^M$, $R(a, x)$ holds iff $a \in A$. Then $p_A(x)$ generates a complete type over $M$ that does not fork over $\emptyset$. Moreover, all nonalgebraic complete types over $M$ extending $\{I(x)\} \cup \bigcup_n\{\lnot I_n(x)\}$ are of this form. 
\item Suppose $B \subseteq I^{M}$ is such that each $|B \cap I^{M}_n| \leq n$. Let $p_B(x)$ be the type over $M$ that says $P(x)$ holds, and $x \not= a$ for each $a \in P^M$, and for each $b \in I^{M}$, $R(x, b)$ holds iff $b \in B$. Then $p_B(x)$ generates a complete type over $M$, and moreover every complete nonalgebraic type over $M$ extending $P(x)$ is of this form. Further, given $M_0 \subseteq M$, we have that $p(x)$ does not fork over $M_0$ iff for each $n < \omega$, $B \cap I_n^{M_0} = B \cap I_n^M$.
\end{itemize}
\end{lemma}

From this lemma we get the following characterization of the saturated models of $T_{cas}$.

\begin{lemma}\label{TCasLemma1v2}
 $M \models T_{cas}$ is $\lambda^+$-saturated iff:

\begin{enumerate}
\item $|I_\alpha| \geq \lambda^+$ for each $\alpha \leq \omega$,
\item For all $B_0, B_1 \subseteq I^M$ disjoint with each $|B_i| \leq \lambda$, and with each $|B_1 \cap I_n| \leq n$, there is $a \in P$ such that $R(a, b)$ for each $b \in B_1$, and $\lnot R(a, b)$ for each $b \in B_0$.
\item For all $A_0, A_1 \subseteq P^M$ disjoint with each $|A_i| \leq \lambda$, there is $b \in I_\omega$ such that $R(a, b)$ for each $a \in A_1$ and $\lnot R(a, b)$ for each $a \in A_0$.
\end{enumerate}
\end{lemma}

Let $\sigma$ be supercompact and let $\lambda \geq \sigma$; fix $\sigma$ and $\lambda$ for the rest of the paper. Write $\mathcal{B} = \mathcal{B}_{2^\lambda, \sigma, \sigma}$ and for each $\alpha < 2^\lambda$ let $\mathcal{B}_\alpha = \mathcal{B}_{\alpha, \sigma, \sigma}$. 

We say that the sequence $(\mathbf{a}_s: s \in [\lambda]^{<\sigma})$ from $\mathcal{B}$ \emph{continuous }if $s \subseteq t$ implies $\mathbf{a}_s \geq \mathbf{a}_t > 0$ and also each $\mathbf{a}_s = \bigcap_{t \in [s]^{<\aleph_0}} \mathbf{a}_t$. We say that $(\mathbf{a}_s: s \in [\lambda]^{<\sigma})$ \emph{is in the filter} $\mathcal{D}$ if each $\mathbf{a}_s \in \mathcal{D}$. Note that if $(\mathbf{a}_s: s \in [\lambda]^{<\aleph_0})$ is a $(\mathcal{B}, \lambda)$ distribution in the $\sigma$-complete filter $\mathcal{D}$, then we automatically get a continuous sequence $(\mathbf{a}_s: s \in [\lambda]^{<\sigma})$.

The following definition is the same as $(\lambda, \sigma, \sigma, \sigma)$-optimality in \cite{Optimals} (where we have set $\mu = \theta = \sigma$).
\begin{definition} 
Suppose $\mathcal{U}$ be an ultrafilter on $\mathcal{B}$. Then $\mathcal{U}$ is $(\lambda, \sigma)$-\emph{optimal} if $\mathcal{U}$ is $\sigma$-complete and, for every continuous sequence $(\mathbf{b}_s: s \in [\lambda]^{<\sigma})$ in $\mathcal{U}$, (A) implies (B):

\begin{itemize}
\item[(A)] There is some closed unbounded $\Omega \subset [\lambda]^{<\sigma}$ such that for each $\delta < 2^\lambda$ with $(\mathbf{b}_s: s \in [\lambda]^{<\sigma}) \subseteq  \mathcal{B}_{\delta}$, and there is some multiplicative refinement $(\mathbf{b}'_s: s \in [\lambda]^{<\sigma})$ of $(\mathbf{b}_s: s \in [\lambda]^{<\sigma})$ from $\mathcal{B}$ so that for each $s \in \Omega$, and for each $\mathbf{a} \in \mathcal{B}_\delta$, if $\mathbf{b}_s \cap \mathbf{a}$ is nonzero then $\mathbf{b}'_s \cap \mathbf{a}$ is nonzero.
\item[(B)] $(\mathbf{b}_s: s \in [\lambda]^{<\aleph_0})$ has a multiplicative refinement $(\mathbf{b}'_s: s \in [\lambda]^{<\aleph_0})$ in $\mathcal{U}$.
\end{itemize}
\end{definition}

Note, in (A) this is the same as saying ``for sufficiently large $\delta$, ..."

The following is Theorem 5.9 from \cite{Optimals}:

\begin{theorem}\label{Th4}
Assuming $\sigma$ is supercompact, there is a $(\lambda, \sigma)$-optimal ultrafilter on $\mathcal{B}$.
\end{theorem}

So to prove our goal, it suffices to establish the following:

\begin{lemma}\label{Th5}
Suppose $\mathcal{U}$ is a $(\lambda, \sigma)$-optimal ultrafilter on $\mathcal{B}$. Then $\mathcal{U}$ is $(\lambda, \mathcal{B}, T_{cas})$-moral.
\end{lemma}
\begin{proof}
Actually, it is easy to check that $T_{cas}$ is $(\lambda, \sigma, \sigma, \sigma)$-explicitly simple, and so we could apply Theorem 7.3 from \cite{Optimals} and be done. For the reader's convenience we give a direct proof, using the niceness of $T_{cas}$.

Choose a regular good filter $\mathcal{D}_0$ on $\mathcal{P}(\lambda)$ and an isomorphism $\mathbf{j}: \mathcal{P}(\lambda)/\mathcal{D}_0 \cong \mathcal{B}$. Write $\mathcal{U}_* = \mathbf{j}^{-1}(\mathcal{U})$. We want to show that $\mathcal{U}_*$ $\lambda^+$-saturates $T_{cas}$.

Let $M \models T_{cas}$, and let $\overline{M} = M^\lambda/\mathcal{D}$; we want to show that $\overline{M}$ is $\lambda^+$-saturated. Since $\mathcal{U}_*$ is $\lambda$-regular, we know that $|I_\alpha^{\overline{M}}| \geq \lambda^+$ for each $\alpha \leq \omega$. So it suffices to realize types $p(x)$ as in items two or three from Lemma~\ref{TCasLemma1v2}. We just consider case two; case three is just easier. So choose $B_0, B_1 \subseteq I^{\overline{M}}$ disjoint with each $|B_i| \leq \lambda$ and each $|B_1 \cap I_n^{\overline{M}}| \leq n$. We show there is $f \in M^I$ such that $[[f/\mathcal{U}_*]] \in P^{\overline{M}}$ and $\overline{M} \models R([[f/\mathcal{U}_*]], b)$ for each $b \in B_1$, and $\overline{M} \models \lnot R([[f/\mathcal{U}_*]], b)$ for each $b \in B_0$. Note that by extending $B_1$, we can suppose each $|B_1 \cap I_n^{\overline{M}}| = n$. So actually we can also suppose that $B_0 \subseteq I_\omega^{\overline{M}}$, as the other elements are redundant.

Enumerate $B_0 \cup B_1 = \overline{b} = (b_\alpha: \alpha < \lambda)$. For each $\alpha < \lambda$ choose $g_\alpha \in (I^M)^\lambda$ with $[[g_\alpha/\mathcal{U}_*]] = b_\alpha$. For each $\alpha < \lambda$, let $i_\alpha$ be such that $b_\alpha \in B_{i_\alpha}$, and let $\gamma_\alpha \leq \omega$ be such that $b_\alpha \in I_{\gamma_\alpha}^{\overline{M}}$. Also, let $\Gamma$ be the set of all $\alpha < \lambda$ such that $b_\alpha \in B_1 \cap \bigcup_{n < \omega} I_n^{\overline{M}}$, i.e. such that $\gamma_\alpha < \omega$; so $\Gamma$ is countable. 

Now, for each formula $\phi(x_0, \ldots, x_{n-1})$ and for each $\alpha_0 < \ldots < \alpha_{n-1}$ let 
$$||\phi(y_{\alpha_0}, \ldots, y_{\alpha_{n-1}})|| := \mathbf{j}(\{i \in I: M \models \phi(g_{\alpha_0}(i), \ldots, g_{\alpha_{n-1}}(i))\}).$$

 For instance, note that each $||I(y_\alpha)|| =1$. For each $s \in [\lambda]^{<\aleph_0}$ let $\mathbf{a}_s = ||\exists x \bigwedge_{\alpha \in s} R(x, y_{\alpha})^{i_\alpha}||$. Then each $\mathbf{a}_s \in \mathcal{U}$, and it suffices to show that $(\mathbf{a}_s: s \in [\lambda]^{<\aleph_0})$ has a multiplicative refinement in $\mathcal{U}$.  Also, let 

$$\mathbf{b}_* = \bigcap_{\alpha \not= \beta \in \Gamma} ||y_\alpha \not= y_\beta|| \, \cap \, \bigcap_{\alpha \in \Gamma} ||I_{\gamma_\alpha}(y_\alpha)||.$$

 So $\mathbf{b}_* \in \mathcal{U}$ since $\mathcal{U}$ is $\sigma$-complete. For each $s \in [\lambda]^{<\aleph_0}$ let $\mathbf{b}_s = \mathbf{b}_* \cap \mathbf{a}_s$. We will show that $(\mathbf{b}_s: s \in [\lambda]^{<\aleph_0})$ has a multiplicative refinement in $\mathcal{U}$.

For each $s \in [\lambda]^{<\sigma}$ let $\mathbf{b}_s := \bigcap_{t \in [s]^{<\aleph_0}} \mathbf{b}_t$; then each $\mathbf{b}_s \in \mathcal{U}$ and $(\mathbf{b}_s: s \in [\lambda]^{<\sigma})$ is continuous. Let $\Omega$ be the club set of all $s \in [\lambda]^{<\sigma}$ with $\Gamma \subseteq s$. It suffices to show that condition (A) holds in the definition of optimality with respect to $\Omega$. So let $\delta$ be given; we can suppose by increasing $\delta$ that whenever $\phi$ is a formula with parameters from $\overline{y} = (y_\alpha: \alpha < \lambda)$, then $||\phi|| \in \mathcal{B}_\delta$.

    Given $s \in [\lambda]^{<\sigma}$ and given $\mathbf{a} \in \mathcal{B}_\delta$ nonzero, say that $\mathbf{a}$ is \emph{strong for} $s$ if for all $\alpha \in s$, if $\beta \leq \alpha$ is least such that $\mathbf{a} \cap  ||y_\alpha = y_{\beta}|| \not= 0$ then $\beta \in s$ and $\mathbf{a} \leq ||y_\alpha = y_\beta||$. 
    
Note that for all $s \in [\lambda]^{<\sigma}$ and for all nonzero $\mathbf{a} \in \mathcal{B}_\delta$, there is $t \supseteq s$ in $[\lambda]^{<\sigma}$ and $\mathbf{b} \leq \mathbf{a}$ such that $\mathbf{b}$ is a nonzero element of $\mathcal{B}_\delta$ and $\mathbf{b}$ is strong for $t$. This is because $\mathcal{B}_\delta$ has a $\sigma$-closed dense subset, and so we can construct $\mathbf{b}$ iteratively. (Recall that $\mathcal{B}_\delta = \mathcal{B}_{\delta, \sigma, \sigma}$ is the Boolean algebra completion of the partial order of functions $P_{\delta, \sigma, \sigma}$; also given $f \in P_{\delta, \sigma, \sigma}$, $\mathbf{x}_f$ is the element of $\mathcal{B}_\delta$ corresponding to $f$. So for our $\sigma$-closed subset of $\mathcal{B}_\delta$ we can take $\{\mathbf{x}_f: f \in P_{\delta, \sigma, \sigma}\}$.)
 
 Suppose $\mathbf{a}$ is strong for $s$. Define $\pi_{\mathbf{a},s}:s \to s$ by $\pi_{\mathbf{a}}(\alpha) = $ the least $\beta \leq \alpha$ with $||y_\alpha = y_\beta|| \cap \mathbf{a}$ nonzero. This $\beta$ is an element of $s$ by definition of strongness, and further we always have $\mathbf{a} \leq  ||y_\alpha = y_\beta||$. Note also that if $\mathbf{a}$ is strong for $s$ and $\mathbf{b}$ is strong for $t$ and $\mathbf{a} \cap \mathbf{b} \not= 0$, then $\pi_{\mathbf{a},s}$ and $ \pi_{\mathbf{b},t}$ agree on $s \cap t$. Further, if $\mathbf{a}$ is strong for $s$, then for each $\alpha, \alpha' \in s$, $\mathbf{c} \leq ||y_\alpha = y_{\alpha'}||$ iff $\mathbf{c} \cap ||y_\alpha = y_{\alpha'}|| \not= 0$ iff $\pi_{\mathbf{c}}(\alpha) = \pi_{\mathbf{c}}(\alpha')$.
 
    For each $s  \in \Omega$ let $\{\mathbf{a}_{s, \xi}: \xi < \xi(s)\}$ and $\{w_{s, \xi} : \xi < \xi(s)\}$ satisfy:
    
    \begin{itemize}
    \item $\{\mathbf{a}_{s, \xi}: \xi < \xi(s)\}$ is a maximal antichain of $\mathcal{B}_\delta$ (and hence of $\mathcal{B}$) below $\mathbf{b}_s$; 
    \item Each $\mathbf{a}_{s, \xi}$ is strong for $w_{s, \xi}$;
    \item Set $\pi_{s, \xi} = \pi_{\mathbf{a}_{s, \xi}, w_{s, \xi}}$. Then $w_{s, \xi} = s \cup \{\pi_{s, \xi}(\alpha): \alpha \in s\}$.
    \end{itemize}
    
  	For each $s, \xi$ define $g_{s, \xi}: w_{s, \xi} \to 2$ by: $g_{s, \xi}(\alpha) =i$ iff there is $\beta \in s$ with $\pi_{s, \xi}(\beta) = \pi_{s, \xi}(\alpha)$ and $i_\beta = i$. This is well-defined since $\mathbf{a}_{s,\xi} \leq \mathbf{b}_s$.

    For each $s, \xi$ define $h_{s, \xi}$ by: $h_{s, \xi} = \{\langle \delta + \alpha, i\rangle : \langle \alpha, i \rangle \in g_{s, \xi}\}$. Then, for each $\alpha < \lambda$ let $\mathbf{b}'_{\{\alpha\}} = \mathbf{b}_* \cap \bigcup \{ \mathbf{a}_{s, \xi} \cap \mathbf{x}_{h_{s, \xi}}: \alpha \in s \in \Omega, \xi < \xi(s)\},$ and for each $s \in [\lambda]^{<\sigma}$ let $\mathbf{b}'_s = \bigcap_{\alpha \in s} \mathbf{b}'_{\{\alpha\}}$. Then it suffices to show $(\mathbf{b}'_s)$ is as in (A) from the definition of optimal ultrafilters. Multiplicativity is clear. Also, suppose $s \in \Omega$ and $\mathbf{a} \in \mathcal{B}_\delta$ is such that $\mathbf{a} \cap \mathbf{b}_s$ is nonzero. Then there is some $\xi < \xi(s)$ such that $\mathbf{a} \cap \mathbf{a}_{s, \xi}$ is nonzero. Then $\mathbf{a} \cap \mathbf{a}_{s, \xi} \cap \mathbf{x}_{h_{s, \xi}}$ is nonzero, but $\mathbf{a}_{s, \xi} \cap \mathbf{x}_{h_{s, \xi}} \leq \mathbf{b}'_s$. 

So it remains to show that for $s \in [\lambda]^{<\sigma}$, $\mathbf{b}'_s \leq \mathbf{b}_s$ . It suffices to show this for finite $s$, since both $(\mathbf{b}'_s)$ and $(\mathbf{b}_s)$ are continuous. Choose $s \in [\lambda]^{<\aleph_0}$ and suppose towards a contradiction that $\mathbf{b}'_s \not \leq \mathbf{b}_s$. Write $\mathbf{c}_0 = \mathbf{b}'_s \cap  -\mathbf{b}_s$ and write $s = \{\alpha_0, \ldots, \alpha_{n-1}\}$. We can inductively choose $\mathbf{c}_0 \geq \ldots \geq \mathbf{c}_n > 0$ such that for each $i < n$, writing $\alpha = \alpha_i$, there is $s_\alpha$ and $\xi_\alpha$ such that $\alpha \in s_\alpha$ and $\mathbf{c}_{i+1} \leq \mathbf{a}_{s_\alpha,\xi_\alpha} \cap \mathbf{x}_{h_{s_\alpha,\xi_\alpha}}$. Let $w = \bigcup_\alpha w_{s_\alpha, \xi_\alpha}$. Also choose $\mathbf{c} < \mathbf{c}_n$ nonzero such that for each $\beta \in w$ and for each $m < \omega$, $\mathbf{c}$ decides $||I_m(y_\beta)||$. This is possible because $\mathcal{B}_\delta$ has a $\sigma$-closed dense subset (and $\sigma > \aleph_0$).
    
    Since the $h_{s_\alpha, \xi_\alpha}$'s are compatible, so must the $g_{s_\alpha, \xi_\alpha}$'s be. Put $g := \bigcup_{\alpha \in s} g_{s_\alpha, \xi_\alpha}$; clearly $g: w \to 2$. Similarly, since $\mathbf{c}$ is strong for each $w_{s_\alpha, \xi_\alpha}$, we have that $\mathbf{c}$ is strong for $w$; let $\pi:= \bigcup_{\alpha \in s} \pi_{s_\alpha, \xi_\alpha}$. 
    
I claim that for each $\beta \in w$ and each $i < 2$, we have that $g(\beta) = i$ iff there is some $\gamma \in \bigcup_{\alpha \in s} s_\alpha$ with $\pi(\beta) = \pi(\gamma)$ and $i_{\gamma} = i$. Left to right is clear by the third condition on the $(\mathbf{a}_{s, \xi}, w_{s, \xi})$'s; for the other direction suppose we had $\gamma \in s_\alpha$ and $\gamma' \in s_{\alpha'}$ with $\pi(\gamma) = \pi(\gamma')$; we want to show that $i_{\gamma} = i_{\gamma'}$. But on the one hand we have $\mathbf{c} \leq ||y_\gamma = y_{\gamma'}||$, and on the other hand, since $\mathbf{c} \leq  \mathbf{b}_{s_{\alpha}} \cap \mathbf{b}_{s_{\alpha'}}$, we have that $\mathbf{c} \leq ||R(x, y_\gamma)^{i_\gamma} \land R(x, y_{\gamma'})^{i_{\gamma'}}||$, from which it follows that $i_{\gamma} = i_{\gamma'}$.
  
  Now, recall that $\mathbf{b}_s = \mathbf{b}_* \cap \mathbf{a}_s$, where $\mathbf{a}_s := ||\exists x \bigwedge_{\alpha \in s} R(x, y_\alpha)^{i_\alpha}||$. Note that $\mathbf{a}_s = ||\phi(y_\alpha: \alpha \in s)||$, where $\phi(y_\alpha: \alpha \in s)$ states that  $y_\alpha \not= y_{\alpha'}$ for each $\alpha, \alpha' \in s$ with $i_\alpha \not= i_{\alpha'}$, and that for each $m < n$ (recall $|s| =n$), there is no $t \in [s]^{m+1}$ such that for each $\alpha \in t$, $I_m(y_\alpha)$ holds and $i_\alpha = 1$, and for each $\alpha, \alpha' \in t$, $y_\alpha \not= y_{\alpha'}$.
  
  By the preceding paragraph, we have that whenever $i_{\alpha} \not= i_{\alpha'}$ we have that $\pi(\alpha) \not= \pi(\alpha')$, hence $\mathbf{c} \leq ||y_\alpha \not= y_{\alpha'}||$. Thus, since $\mathbf{c} \cap \mathbf{b}_s = 0$, there must be some $m < n$ and some $t \in [s]^{m+1}$  such that for each $\alpha \in t$, $I_m(y_\alpha)$ holds and $i_\alpha = 1$, and for each $\alpha, \alpha' \in t$, $y_\alpha \not= y_{\alpha'}$. Let $t'$ be the set of all $\beta < \lambda$ such that $\gamma_\beta = m$; so $t' \in [\Gamma]^m$, in particular $t' \subseteq s_\alpha$ for each $\alpha \in s$ (since each $s_\alpha \in \Omega$). By the pigeonhole principle, choose $\alpha_* \in t$ such that $\pi(\alpha_*) \not= \pi(\beta)$ for each $\beta \in t'$. But then $\mathbf{c} \leq ||R(x, y_\alpha)||$ for each $\alpha \in t' \cup \{\alpha_*\}$, and $\mathbf{c} \leq ||I_m(y_\alpha)||$ for each $\alpha \in t' \cup \{\alpha_*\}$, contradicting that $\mathbf{c}$ is nonzero.
\end{proof}

\begin{theorem}
Suppose there is a supercompact cardinal. Then Keisler's order is not linear.
\end{theorem}
\begin{proof}
By Theorems~\ref{Th1} and \ref{Th2} we know that $T_{cas} \not \trianglelefteq T_{n, k}$, and by Theorems~\ref{Th3} and \ref{Th4} and Lemma~\ref{Th5} we conclude that $T_{n, k} \not \trianglelefteq T_{cas}$.
\end{proof}
  
\end{document}